\documentclass[a4paper,11pt,oneside]{amsart}

\usepackage{geometry}
\usepackage{graphicx, amsfonts, amssymb}
\usepackage{mathrsfs, amsmath, amsthm}
\usepackage{verbatim}
\usepackage{relsize}
\usepackage{enumerate}
\usepackage{hyperref}
\usepackage{version}
\usepackage{blindtext}
\usepackage{mathtools}

\newtheorem{theorem}{Theorem}[section]
\newtheorem*{thma}{Theorem A}

\newtheorem{Lemm}[theorem]{Lemma}

\theoremstyle{remark}

\newcommand{\m}{\mathbb}
\newcommand{\mm}{\mathcal}

\DeclareMathOperator{\Hol}{Hol}

\hypersetup{
	colorlinks=true,
	linkcolor=blue,
	filecolor=magenta,      
	urlcolor=cyan,
	pdftitle={Trace Ideal Criteria for Generalized Integration Operators},
	pdfpagemode=FullScreen,
}

\title{Trace Ideal Criteria for Generalized Integration Operators}

\author{G. Nikolaidis}
\email{nikolaidg@math.auth.gr}
\address{Department of Mathematics, Aristotle University of Thessaloniki, 54124, Greece}

\thanks{This research project was supported by the Hellenic Foundation for Research and Innovation (H.F.R.I.) under the `2nd Call for H.F.R.I. Research Projects to support Faculty Members \& Researchers' (Project Number: 4662).}

\date{}

\begin{document}

\begin{abstract}
 For $g\in \Hol(\m D)$, we study the class of generalized integration operators $T_{g,a}$, acting on Hardy and Bergman spaces of the unit disc in the complex plane. This class of integral operators were introduced to study factorization theorems in the Hardy spaces of the unit disc. We completely characterize the space of symbols $g$, for which  $T_{g,a}$ belongs to the Schatten-von Neumann ideals of the Hardy and Bergman spaces.
\end{abstract}
\maketitle
\section{Introduction}
Let $\m{D}$ be the unit disc in the complex plane. By $\Hol(\m{D})$, we denote the space of all analytic functions defined in $\m{D}$. Let 
\begin{equation}\label{Voltera operator definition}
    Vf(z)=\int_{0}^{z}f(\zeta)\,d\zeta\qquad f\in\Hol(\m D)
\end{equation}
be the classical Volterra operator, acting on $\Hol(\m{D})$. For a fixed symbol $g\in\Hol(\m{D})$, we can define the operator
\begin{equation}\label{equation for Tg}
    T_gf(z)=\int_{0}^{z}f(\zeta)g'(\zeta)\,d\zeta\,,\qquad z\in\m{D}\,, f\in \Hol(\m D).
\end{equation}
 The study of integral operators such as $T_g$, originated from the fact that, as $g$ varies in $\Hol(\m{D})$, $T_g$ represents some significant classical operators. 
For instance, if $g(z)=z$, $T_g$ is the  Volterra operator $V$, while, when $g(z)=-\log(1-z)$, it coincides with the (shifted) Ces\'aro summation operator \cite{siskakis1987composition}. It is the above fact, which justify that $T_g$ operators are also known as generalized Volterra operators. 
  \par Originally, the class of generalized Volterra operators was introduced and studied in the context of the Hardy and Bergman spaces of the unit disc. The Hardy space $H^p$ is the space of all $f\in \Hol(\mathbb{D})$ such that

$$   \sup_{0\leq r<1}\int_{0}^{2\pi}|f(re^{it})|^p\,\frac{dt}{2\pi}<\infty\,,$$
while for $\alpha>-1$, the standard weighted Bergman space $A_\alpha^p$ is the space of all $f\in\Hol(\m{D})$ such that
 $$\|f\|_{A_\alpha^p}^p=\int_{\m{D}}|f(z)|^pdA_\alpha(z)<\infty\,,$$
 where $dA(z)$ is the normalized Lebesgue measure on $\m{D}$ and $dA_\alpha(z)=(\alpha+1)(1-|z|^2)^{\alpha}dA(z)$. 
	
	 In particular, Ch. Pommerenke \cite{pommerenke1977schlichte}, characterized the space of symbols $g$ for which $T_g$ is bounded on $H^2$. The complete characterization of symbols for which $T_g$ acts boundedly between different Hardy spaces was given in a series of papers \cite{aleman1995integral},\cite{aleman2001integral}. Subsequently, the study of such operators in various spaces of analytic functions has gained since considerable attention (see for example \cite{Wu2011Volterra}, \cite{pelaez2014weightedmemoir}, \cite{TanausuGalanopoulostentspaces},\cite{GomezCabello2025Volterra}, \cite{miihkinen2020volterra}).
     \par In the initial study of generalized Volterra operators, A. Aleman and A. Siskakis \cite{aleman1997integration},\cite{aleman1995integral} established sufficient and necessary conditions on the symbol $g\in\Hol(\m{D})$ for $T_g$ to belong to the Schatten-von Neumann ideals in the case of the Hilbert space $H^2$ and $A^2_\alpha$, when $\alpha>-1$. For the sake of notation, we write heuristically $H^2$ by $A^2_{-1}$. 
     \par In particular, for  $\alpha\geq -1$, $g\in \Hol(\m{D})$ and $1<p<\infty$, we have that $T_g\in S^p(A^2_\alpha)$ if and only if $g$ belongs to the Besov space $\mm{B}_p$, while if $0<p\leq 1$ and $T_g\in S^p(A^2_\alpha)$, then $T_g$ is the zero operator.
        
     For $0<p<\infty$ and $m\in\m{N}$ such that $mp>1$, the Besov spaces $\mm{B}_p$ contain all $f\in \Hol(\m{D})$ such that
 \begin{equation*}
     \int_{\m{D}}|f^{(m)}(z)|^p(1-|z|^2)^{mp}\frac{dA(z)}{(1-|z|^2)^2}<\infty\,.
 \end{equation*}
  The Besov spaces $\mm{B}_p$ do not depend on $m$, as shown in \cite[Lemma 5.16 and Theorem 4.28]{zhu2007operator} and the above quantity defines a seminorm on $\mm{B}_p$.
 	\par In this article, we study a further generalization of the integral operator $T_g$. 
An inspection of (\ref{equation for Tg}) shows that $T_g$ is a primitive of the first term of the derivative of the product $f g$. 
Applying the Leibniz rule of differentiation we get 
\begin{equation}\label{eq:Leibniz_dif}
 (T_gf)^{(n)}=\sum_{k=0}^{n-1}\binom{n-1}{k}f^{(k)} g^{(n-k)}.
\end{equation}

	Now, if we consider an arbitrary $n$-tuple $a=(a_0,a_1,\dots,a_{n-1})\in \mathbb{C}^{n}$, $a\neq \mathbf{0}$ and a symbol $g\in \Hol(\m D)$, we can define the following operator
 \begin{equation*}\label{equation defining Tga}
     T_{g,a}f=V^n\left(\sum_{k=0}^{n-1}a_kf^{(k)} g^{(n-k)}\right)\quad f\in \Hol(\m D)\,,
 \end{equation*}
 where $V^n$ is the $n$-th iterate of the Volterra operator (\ref{Voltera operator definition}). 
 It is then clear by \eqref{eq:Leibniz_dif} that the generalized Volterra operator is 
 a particular instance of the operator $T_{g,a}.$ 
 \par  The integral operator $T_{g,a}$ was introduced by N. Chalmoukis, and studied in the context of Hardy spaces of the unit disc. In his work \cite{chalmoukis2020generalized}, he characterized the space of symbols 
$g\in \Hol(\m D)$ for which $T_{g,a}\colon H^p \to H^q$ is bounded, for $0<p\leq q<\infty$. The case of $0<q<p<\infty$ was later resolved by the author and Chalmoukis in \cite{NikolaidisChalmoukis2024}. Around the same time, H. Arroussi et al. \cite{arroussi2024new} characterized Sobolev-Carleson measures for Bergman spaces and, as a consequence, characterized the space of symbols $g$ for which $T_{g,a}$ is bounded and compact between Bergman spaces. Since then, significant research has been devoted to identifying the space of symbols $g\in\Hol(\m{D})$, for which $T_{g,a}$ is bounded or compact on other spaces of analytic functions. Readers interested in further developments may refer to  \cite{tong2024generalizedvolterratypeintegraloperators},\cite{ArroussiTgafockspaces},\cite{SongxiaoLitentspacesTga}.
 \par The motivation for studying the operator $T_{g,a}$, extends beyond its generalization of the classical $T_g$. It also arrises from its connection to factorization theorem of holomorphic functions by W.Cohn \cite{cohn1999factorization} and a theorem of J. R\"atty\"a \cite{jounirattya2007lineareq} about higher order linear differential equations with holomorphic coefficients. For further details, we refer the interested reader to \cite[Theorems 1.5 and 1.6]{chalmoukis2020generalized}. Notably, the latter result has been extended to other spaces of analytic functions as well (see \cite{tong2024generalizedvolterratypeintegraloperators},\cite{ArroussiTgafockspaces},\cite{arroussi2024new}).
  
  \par Our study focuses on providing sufficient and necessary conditions on the symbol $g\in\Hol(\m{D})$ for $T_{g,a}$ to belong to the Schatten-von Neumann ideal on Hardy and Bergman spaces.  To achieve this, we first examine the case where $a$ is a standard unit vector of $\m{C}^n,$ a situation previously addressed from J. Du, S. Li and D. Qu \cite{DuLiQu2022generalized}. In this setting, the operators $T_{g,a}$ have the following form
 \begin{equation*}\label{Tgnk form equation}
     T_g^{n,k}(f)(z)=V^n(f^{(k)}g^{(n-k)})(z)\qquad 0\leq k< n.
 \end{equation*} 
 Their result is the following.
  \begin{thma}\hypertarget{TheoremLI}{L}et $\alpha\geq -1$, $n\in\m{N}$, $0\leq k\leq n-1$, $g\in\Hol(\m{D})$ and $0<p<\infty$.
      \begin{itemize}
         \item[(i)] If $\frac{1}{n-k}<p<\infty$, then $T_{g}^{n,k}\in S^p(A^2_\alpha)$ if and only if $g\in \mm{B}_p$.
         \item[(ii)] If $0<p\leq \frac{1}{n-k}$ and $T_{g}^{n,k}\in S^p(A^2_\alpha)$, then $T_{g}^{n,k}$ is the zero operator.
     \end{itemize}
  \end{thma}
  We present now the main result of this article.
 \begin{theorem}\label{Theorem of paper Schatten class membership}
     Let $\alpha\geq -1$, $n\in\m{N}$, $a=(a_0,\dots,a_{n-1})\in\m{C}^n\setminus\{\mathbf{0}\}$, $\lambda=\max\{k\colon a_k\neq 0\}$, $g\in \Hol(\m{D})$,  and $0<p<\infty$. 
     \begin{itemize}
         \item[(i)] If $\frac{1}{n-\lambda}<p<\infty$, then $T_{g,a}\in S^p(A^2_\alpha)$ if and only if $g\in \mm{B}_p$.
         \item[(ii)] If $0<p\leq \frac{1}{n-\lambda}$ and $T_{g,a}\in S^p(A^2_\alpha)$, then $T_{g,a}$ is the zero operator.
     \end{itemize}
 \end{theorem}
 \par When $T_{g,a}=T_g$ and then $\lambda=n-1$, we see that Theorem \ref{Theorem of paper Schatten class membership} reduces to the already known result of A. Aleman and A. Siskakis. We highlight also that our work focuses on proving that $g\in\mm{B}_p$ is necessary for $T_{g,a}\in S^p(A^2_\alpha)$. Indeed, since Schatten-von Neumann ideals are vector spaces, Theorem \hyperlink{TheoremLI}{A} implies that if $g\in\mm{B}_p$, then $T_{g,a}\in S^p(A^2_\alpha)$. Actually, another interpretation of our result is that if $T_{g,a}\in S^p(A^2_\alpha)$, then every term comprising the operator is forced to be in $S^p(A^2_\alpha)$ as well. In other words, there is no cancellation between the terms of $T_{g,a}$. 
 \par The result concerning the Schatten-von Neumann ideal membership for the operator $T_g$ relies primarily on the result of Luecking \cite{LUECKINGtraceidealcriteriatoeplitz}, which provides sufficient and necessary conditions for which Toeplitz operators belong to the Schatten-von Neumann ideals. Due to this connection, we follow his approach while also incorporating known techniques to handle possible cancellations between the terms of $T_{g,a}$, which  ultimately leads to the proof of Theorem \ref{Theorem of paper Schatten class membership}.

\par The article is constructed as follows. In the following section, we provide the necessary background material we require for the hyperbolic metric on $\m{D}$ and Schatten-von Neumman ideals, while also introducing properly the spaces of analytic functions which are involved. In Section 3, we fix $\alpha\geq -1$ and we prove Theorem \ref{Theorem of paper Schatten class membership} for $2\leq p<\infty$. The final section addresses the case of $0<p<2$.

\section{Preliminaries}

\subsection{Hyperbolic geometry}We recall some elementary facts from the hyperbolic geometry of the unit disc. We use the notation
 $$\rho(z,w)=\left|\frac{z-w}{1-\overline{z}w}\right|,\qquad z,w\in\m{D}$$
 for the {\it pseudohyperbolic} metric of $\m{D}$. Moreover,
 $$\beta(z,w)=\frac{1}{2}\log\frac{1+\rho(z,w)}{1-\rho(z,w)},\qquad z,w\in\m{D}$$
 is the {\it hyperbolic} metric of $\m{D}$.
  The set $\displaystyle{D(z,r)= \{w\in\m{D}\colon \beta(z,w)<r\}}$ is the hyperbolic disc, centered at $z$ with radius $r>0$. It is known that the modulus of analytic function satisfy sub-mean inequalities for hyperbolic discs. In fact, for every  $0<p<\infty$ and $r>0$ , there exists a constant $C$, depending only on $r$,  such that for all $z\in\m{D}$ and $f\in \Hol(\mathbb{D})$, 
  \begin{equation}\label{subharmonicity hyperbolic estimate}
     |f(z)|^p\leq \frac{C}{(1-|z|^2)^{2}}\int_{D(z,r)}|f(w)|^p\,dA(w).
 \end{equation}
 This inequality is proved in  \cite[Lemma 13, p. 66]{duren2004bergman}. It is also well known, see \cite[Proposition 4.5]{zhu2007operator}, that given $z\in\m{D}$ and $r>0$, there exists constants $C_1,C_2,C_3$, depending on $r$, such that
\begin{align}
    \frac{1}{C_1}|1-\overline{w}a|&\leq |1-\overline{w}z|\leq C_1|1-\overline{w}a|\,,\qquad w\in \m{D},a\in D(z,r),\nonumber\\
    \frac{1}{C_2}(1-|w|^2)&\leq 1-|z|^2\leq C_2(1-|w|^2)\,,\qquad w\in D(z,r),\label{geometric hyperbolic estimates}\\
    \frac{1}{C_3}|D(z,r)|&\leq (1-|z|^2)^2\leq C_3|D(z,r)|.\nonumber
\end{align}
where $|D(z,r)|$ denotes the Lebesgue volume of the set $D(z,r)$. We refer to the estimates of the form \eqref{geometric hyperbolic estimates} as hyperbolic estimates and shall be used repeatedly in the proofs of Theorem \ref{Theorem of paper Schatten class membership}.
\par A sequence $\{z_\lambda\}_\lambda\subset \m{D}$ is called $r$-{\it hyperbolically  separated} if there exists  a constant $r>0$ such that $\beta(z_k,z_\lambda)\geq r $ for $k\neq\lambda$, while is said to be an $r$-{\it lattice} in the hyperbolic distance, if it is $r/2$ - separated and  
$$\mathbb{D} = \bigcup_k D(z_k, r).$$
Hyperbolically separated sequences have the following useful ``finite overlapping'' property.

\begin{Lemm}\label{finite covering Lemma}
    Let $r>0$ and $Z=\{z_\lambda\}_\lambda$ be a $r$-hyperbolically separated sequence. There exists a positive constant $N$, depending only on $r$, such that for every $z\in\m{D}$ there exist at most $N$ hyperbolic discs $D(z_\lambda,r)$ such that $z\in D(z_\lambda,r).$  
\end{Lemm}
A proof of the above result can be found in \cite[Lemma 4.7]{zhu2007operator}.
\subsection{Hilbert Spaces of analytic functions} In this article, the main Hilbert spaces under consideration are the standard weighted Bergman spaces and the Hardy space. For convenience, we adopt the heuristic notation $A^2_{-1}$  to denote the Hardy space $H^2$. For fixed $\alpha\geq -1$, the inner product on these spaces is given by
\begin{align*}\langle f,h\rangle_{A^2_\alpha}&=\sum_{n=0}^{\infty}\frac{n!\Gamma(\alpha+2)}{\Gamma(\alpha+n+2)}a_n\overline{b_n}\,.
\end{align*}
where $\Gamma$ is the classical gamma function and $a_n,b_n$ are the Taylor coefficients of $f,h$ respectively. A consequence of Parseval's identity for analytic functions, allow us to compute the inner product using an integral involving the derivatives of $f,h$. This technique is the well known {\it Littlewood-Paley formula}. In fact, given $n>0$ and $f\in \Hol(\m D)$ such that $f(0)=\dots =f^{(n-1)}(0)=0$, there exists a constant $C$, depending only on $n$, such that
\begin{align}
    \frac{1}{C}\|f\|_{A^2_\alpha}^2\leq \int_{\m{D}}|f^{(n)}(z)|^2(1-|z|^2)^{2n+\alpha}dA(z)\leq C\|f\|_{A^2_\alpha}^2\,.\label{Littlewood Paley Bergman}
\end{align}
For $n=1$ and $\alpha=-1$, we refer to the proof of the above result in \cite[Lemma 3.2, Ch. VI \textsection 3]{garnett2006bounded}. For general $\alpha>-1$, we refer to \cite[Theorem 4.28]{zhu2007operator}. For an arbitrary $n>1$, one works by induction on \eqref{Littlewood Paley Bergman} to obtain the result.
\par As our operators are induced by analytic functions, we require suitable test functions to transfer the information of Schatten $p$-norm into conditions involving only the function $g\in \Hol(\m{D})$. We introduce these functions below. Let $i\in\m{Z}_+$, and $\gamma>1+\frac{\alpha+2}{2}$ be sufficiently large. We consider the functions 
$$K_{j,
w}(z) = \frac{z^j}{(1-\overline{w}z)^{\gamma+j}} \qquad z\in\m{D}\,.$$
A key estimate, used repeatedly in the proof of Theorem \ref{Theorem of paper Schatten class membership}, states that for $n\in\m{Z}_+$, we can find a constant $C$, depending only on $n,\gamma$ such that
\begin{equation}\label{equation to use for Kiw}
    |K_{j,w}^{(n)}(z)|\leq \frac{C}{|1-\overline{w}z|^{\gamma+j+n}}\,.
\end{equation}
This class of functions was introduced by Arroussi et al. \cite{arroussi2024new} to solve a system of linear equations related to the boundedness of the operator $T_{g,a}$ in Bergman spaces, a technique that we also employ in this article.
\par To proceed, we state a consequence of the atomic decomposition of functions $A^2_\alpha$, $\alpha>-1$, written in the language of our test functions.
\begin{Lemm}\label{Operator Lhmma pgeq 2}
    Let $\alpha\geq -1$, $\gamma>1+\frac{a+2}{2}$ be sufficiently large, $0\leq j\leq n-1$ and $\{z_n\}_n$ be a $r$-hyperbolically separated sequence in $\m{D}$ and $\{c_n\}_n\subset \ell^2$\,. Then, the functions
    \begin{align*}
    A_j(z)&=\sum_{n=1}^{\infty}c_n\frac{(1-|z_n|^2)^{\gamma-\frac{\alpha+2}{2}}}{(1-|z_n|^j\overline{z_n}z)^\gamma}\qquad z\in\m{D}\\    
    B_j(z)&=\sum_{n=1}^{\infty}c_n(1-|z_n|^2)^{{\gamma+j}-\frac{\alpha+2}{2}}K_{j,z_n}(z)\qquad z\in\m{D}\,,
    \end{align*}
    belong to $A^2_\alpha$ and there exists a positive constant $C$, depending only on $\alpha,\gamma,r$, such that
    $$\|A_j\|_{A^2_\alpha}^2,\|B_j\|_{A^2_\alpha}^2\leq C\sum_{n=1}^{\infty}|c_n|^2\,.$$
\end{Lemm}
This result follows from a slight modification of a special case of \cite[Lemma 6]{LUECKINGtraceidealcriteriatoeplitz}. In the modern literature, the aforementioned result is also proved in \cite[Theorem 4.33]{zhu2007operator}.
\subsection{Schatten-von Neumann ideals} We formally introduce the spaces of operators considered in this article. Throughout, we fix a separable Hilbert space, equipped with inner product $\langle \cdot,\cdot\rangle$ and denote its induced norm by $\|\cdot \|$. We use the same notation for the operator norm of a bounded operator acting on $\mm{H}$.
\par For a compact operator $T$ on $\mm{H}$, there exists a decreasing sequence of positive numbers $\{\lambda_n(T)\}_n$ and orthonormal sets $\{e_n\}_n,\{\sigma_n\}_n\subset \mm{H}$ such that
$$T(x)=\sum_{n=1}^{\infty}\lambda_n(T)\langle x,e_n\rangle \sigma_n\qquad x\in\mm{H}\,.$$
The sequence $\{\lambda_n(T)\}_n$ consists of the eigenvalues of $|T|=(T^*T)^{\frac{1}{2}}$, repeated according to their multiplicity, and $\{e_n\}_n$ are the corresponding normalized eigenvectors. The sequence of numbers $\{\lambda_n(T)\}_n$ is called the {\it singular value sequence of} $T$. An equivalent characterization of the singular values is given by the {\it Rayleigh's equation}. Let $n\in\m{N}$, and $\lambda_n(T)$ be the $n$-th singular value of the positive operator $T$. Then,
\begin{equation}\label{Min-Max Theorem}
    \lambda_{n+1}=\min_{x_1,\dots,x_n}\max\{\langle Tx,x\rangle\colon \|x\|=1, x\perp x_i, 1\leq i\leq n\}\,. 
\end{equation}
A proof of this result can be found in \cite[Ch. X.4.3]{dunfordschwarzt1988linear}.
\par For $0<p<\infty$, the Schatten-von Neumann ideal of $\mm{H}$, denoted by $S^p(\mm{H})$, is the space of all compact operators $T$ acting on $\mm{H}$, such that the singular value sequence of $T$ belongs to $\ell^p$. Let
\begin{equation*}
\|T\|_p=\left(\sum_{n=1}^{\infty}\lambda_n(T)^p\right)^{\frac{1}{p}}\,.
\end{equation*}
For $1\leq p<\infty$, $(S^p(\mm{H}),\|\cdot \|_p)$ becomes a Banach space, while for $0<p<1$, $\|\cdot \|_p^p$ defines a metric and $S^p(\mm{H})$ is a complete topological space. Standard textbooks covering the topic of Schatten ideals are \cite{gohberg1978introduction}, \cite{ringrose1971compact}.
\par For $0<p<\infty$, we use the fact that $S^p(\mm{H})$ is a two-sided ideal in the ring of bounded linear operators acting on $\mm{H}$. In particular, if $T\in S^p(\mm{H})$ and $A,B$ are bounded operators on $\mm{H}$, then $ATB\in S^p(\mm{H})$ with
\begin{equation}\label{norm inequality when composed in Schatten}
    \|ATB\|_p\leq \|A\|\|T\|_p\|B\|\,.
\end{equation}
A proof of the above result can be found in \cite[Ch. XI Lemma 9]{dunfordschwarzt1988linear}. Moreover, the Spectral Mapping Theorem imply that the condition $T\in S^p(\mm{H})$ is equivalent to $T^*T\in S^{\frac{p}{2}}(\mm{H})$ and
\begin{equation}\label{equation T*T norm}
    \|T\|_p^p=\|T^*T\|_{p/2}^{p/2}\,.
\end{equation}
In the following sections, we estimate the Schatten norm of the $T_{g,a}$ operator with a quantity containing its norm on orthonormal bases. We present the result below.
\begin{Lemm}\label{Lemma of Schatten norm}
    Let $T$ be a compact operator acting on $\mm{H}$. The following conditions hold,
    \begin{itemize}
        \item[(i)] If $2\leq p<\infty$, then
        $$\|T\|_p^p=\max\biggl\{\sum_{n=1}^{\infty}\|T(e_n)\|^p\colon \{e_n\}_n \text{ orthonormal base}\biggr\}.$$
        \item[(ii)] If $0<p\leq 2$, then
        $$\|T\|_p^p=\min\biggl\{\sum_{n=1}^{\infty}\|T(e_n)\|^p\colon \{e_n\}_n \text{ orthonormal base}\biggr\}.$$
    \end{itemize}
    Both max and min are attained when $\{e_n\}_n$ is chosen to be the sequence of the eigenvectors corresponding to the singular value sequence $\{\lambda_n(T^*T)\}_n$.
\end{Lemm}
One can prove this result by applying the results of \cite[Theorem 1.27]{zhu2007operator} and \cite[Proposition 1.31, Corollary 1.32]{zhu2007operator} to $T^*T$. This leads to two significant consequences. First, it makes necessary the distinction between two separate proofs for Theorem \ref{Theorem of paper Schatten class membership}, one for the case $p\geq 2$ and one for $0<p<2$. Second, it establishes that the membership of an operator in a Schatten-von Neumann ideal can be determined using an equivalent norm on the Hilbert space, rather than the norm originally equipped on it. This observation plays a fundamental role in allowing the computation of these norms equivalently via the Littlewood-Paley formula \eqref{Littlewood Paley Bergman}. \par Finally, for $0<p<2$, we shall require an estimate from above of the Schatten norm, proved in \cite[Proposition 1.29]{zhu2007operator}.
\begin{Lemm}\label{0<p<2 Schatten norm estimation}
    Let $T$ be a compact operator acting on $\mm{H}$ and $0<p\leq 2$. Then, for any orthonormal basis $\{e_n\}_n\subset \mm{H}$, we have
    $$\|T\|_{p}^p\leq \sum_{n=1}^{\infty}\sum_{k=1}^{\infty}|\langle T(e_n),e_k\rangle|^p\,.$$
\end{Lemm}

\subsection{Some further lemmas}
Finally, we require some technical results to overcome the possible cancellations between the terms consisting the $T_{g,a}$ operator. The first one is proved in \cite[Lemma 2.9]{NikolaidisChalmoukis2024}.
 \begin{Lemm}\label{algebraic lhmma}
    Let $f_0,f_1,\dots,f_{n-1}$ be complex valued functions on the unit disc. Given the system of linear equations
    $$D_j(z) = \sum_{k=0}^{n-1}|z|^{jk}f_k(z)\frac{(1-|z|^2)^n}{(1-|z|^{j+2})^k}\qquad j=0,\dots,n-1\,,$$
    then for each $0\leq k\leq n-1$,
    $$f_k(z)(1-|z|^2)^{n-k}=\sum_{j=0}^{n-1}b_{jk}(z) D_j(z),\qquad 0<
    |z|<1,$$
    where $b_{jk}$ are bounded when $\frac{1}{2}<|z|<1.$
\end{Lemm}
 For the second result, we highlight that is a specific case of \cite[Lemma 2.5]{arroussi2024new}. The reader can optimize the result as Gaussian elimination of rows, as explained in the proof of \cite[Proposition 2.6]{arroussi2024new}. He can verify also its strength in the proofs of \cite[Theorem 1.1 and Theorem 1.2]{arroussi2024new}
\begin{Lemm}\label{Lemma of Arrousi on linear algebra}
Let $f_0,f_1,\dots,f_{n-1}$ be complex valued functions on the unit disc and $w\in \m {D}$. Given the linear system
$$D_{w,j}(z)=\sum_{k=0}^\lambda f_k(z)(1-\overline{w}z)^jK_{w,j}^{(k)}(z)\,,$$
there exists functions $b_{j}$, $0\leq j\leq \lambda$ bounded in $\m{D}$, such that
$$\frac{f_{\lambda}(z)}{(1-\overline{w}z)^{\gamma+\lambda}}=\sum_{j=0}^{\lambda}b_j(z)D_{w,j}(z)\,.$$
\end{Lemm}

  \section{Proof of Theorem \ref{Theorem of paper Schatten class membership} for \texorpdfstring{$2\leq p<\infty$}{}}

 \begin{proof}[Proof of Theorem \ref{Theorem of paper Schatten class membership} for $2\leq p<\infty$]
  Let $\alpha\geq -1$. Fix $0\leq j\leq n-1$ and $r>0$. Consider  $\{z_n\}_n$ a $r$-lattice on $\m{D}$ and set $$G(z,w)=\mathlarger{\mathlarger{\sum_{k=0}^{n-1}}}\frac{a_k(b)_k|w|^{jk}\overline{w}^kg^{(n-k)}(z)}{(1-|w|^j\overline{w}z)^k}\qquad z,w\in\m{D}\,.$$
 
 Let 
$\mathcal{W}=\{w_n\}_{n}$ be a sequence of points such that
\begin{itemize}
    \item[i)] $w_n\in \overline{D(z_n,r)}$;
    \item[ii)] $w_n$ is a point where the function ${\displaystyle |G(z,z)|(1-|z|^2)^{n}}$ takes its maximum value in $\overline{D(z_n,r)}.$
\end{itemize}
$\mathcal{W}$ may not be hyperbolic separated, as the hyperbolic discs $D(z_n,r)$ may have non-trivial intersections. However, we can apply Lemma \ref{finite covering Lemma}, to partition the sequence $\{z_n\}_n$ into $N$ disjoint subsequences such that within each subsequence, the hyperbolic discs $D(z_n,r)$ are pairwise disjoint and $N$ depending only on $r$. Consequently, the corresponding sequence $\{w_n\}_n$ can also be decomposed into $N$ disjoint subsequences, each of which is hyperbolically separated. In the following, for simplicity, we denote again as $\mathcal{W}$ each such corresponding $r$-hyperbolically separated subsequence of $\mm{W}$.
    \par Now, assuming that $T_{g,a}\in S^p$, we equivalently have $T=(T_{g,a})^*T_{g,a}\in S^{\frac{p}{2}}(A^2_\alpha)$ with the corresponding Schatten norms satisfying \eqref{equation T*T norm}. Fix an orthonormal base $\{e_n\}_n\subset A^2_\alpha$ and define the following operator on $A^2_\alpha$, given by 
    
    $$A_j(f)(z)=\sum_{n=1}^{\infty}\langle f,e_n\rangle_{A^2_\alpha}\frac{(1-|w_n|^2)^{\gamma-\frac{\alpha+2}{2}}}{(1-|w_n|^j\overline{w_n}z)^\gamma}\qquad f\in A^2_\alpha\,.$$
    According to Lemma \ref{Operator Lhmma pgeq 2} and Parseval's identity, $A_j$ is bounded on $A^2_\alpha$. Consequently, according to \eqref{norm inequality when composed in Schatten}, we have
    $T_j=A_j^*TA_j\in S^{\frac{p}{2}}(A^2_\alpha)$ and there exists a constant $C$, depending only on $\gamma,a,r$, such that
    $$\|T_j\|_{p/2}^{p/2}\leq C\|T_{g,a}\|_{p}^p\,.$$
    Then, by the means of Lemma \ref{Lemma of Schatten norm} and \eqref{Littlewood Paley Bergman}, there exists a constant $C$, depending only on $n$, such that
\begin{align*}
\|T_j\|_{p/2}^{p/2}&\geq \sum_{\lambda=1}^{\infty}\|T_{g,a}(A_j(e_\lambda))\|_{A^2_\alpha}^{p}\\
   &\geq C\sum_{\lambda=1}^{\infty}\left(\int_{\m D}
   \left|\sum_{k=0}^{n-1}a_k\frac{(b)_k|w_\lambda|^{jk}\overline{w_\lambda}^k(1-|w_\lambda|^2)^{b-\frac{\alpha+2}{2}}}{(1-|w_\lambda|^j\overline{w_\lambda}z)^{b+k}}g^{(n-k)}(z)\right|^2dA_{2n+\alpha}(z)\right)^{p/2}\\
   &\geq C\sum_{\lambda=1}^{\infty}\left(\int_{D(w_\lambda,r)}
   \left|\sum_{k=0}^{n-1}a_k\frac{(b)_k|w_\lambda|^{jk}\overline{w_\lambda}^k(1-|w_\lambda|^2)^{b-\frac{\alpha+2}{2}}}{(1-|w_\lambda|^j\overline{w_\lambda}z)^{b+k}}g^{(n-k)}(z)\right|^2dA_{2n+\alpha}(z)\right)^{p/2}
\end{align*}
Henceforth, with the help of hyperbolic estimates \eqref{geometric hyperbolic estimates}, we are able to find another constant $C$, this time depending on $n,r$, such that
\begin{align}\label{equation to reach pgeq2}
    \|T_j\|_{p/2}^{p/2}
   &\geq C\sum_{\lambda=1}^{\infty}\left(\int_{D(w_\lambda,r)}
   \left|\sum_{k=0}^{n-1}a_k\frac{(b)_k|w_\lambda|^{jk}\overline{w_\lambda}^k}{(1-|w_\lambda|^j\overline{w_\lambda}z)^{k}}g^{(n-k)}(z)\right|^2(1-|w_\lambda|^2)^{2n-2}dA(z)\right)^{p/2}\nonumber
\end{align}
The function $G(z,w_\lambda)$ is subharmonic in the first variable, hence by subharmonicity  over the hyperbolic disc $D(w_\lambda,r)$, see \eqref{subharmonicity hyperbolic estimate},  we can find a positive constant $C$, depending on $r$, such that

\begin{align*}
    \left|\sum_{k=0}^{n-1}a_k\frac{(b)_k|w_\lambda|^{jk}\overline{w_\lambda}^k}{(1-|w_\lambda|^{j+2})^{k}}g^{(n-k)}(w_\lambda)\right|^p&\leq C\left(\int_{D(w_\lambda,r)}
   \left|\sum_{k=0}^{n-1}a_k\frac{(b)_k|w_\lambda|^{jk}\overline{w_\lambda}^k}{(1-|w_\lambda|^j\overline{w_\lambda}z)^{k}}g^{(n-k)}(z)\right|^2\frac{dA(z)}{(1-|w_\lambda|^2)^2}\right)^{p/2} 
\end{align*}
We combine all the aforementioned estimates, to find a constant $C$, depending on $n,r$ such that
$$\|T_j\|_{p/2}^{p/2}\geq C \sum_{\lambda=1}^{\infty} \left|\sum_{k=0}^{n-1}a_k\frac{(b)_k|w_\lambda|^{jk}\overline{w_\lambda}^k}{(1-|w_\lambda|^{j+2})^{k}}g^{(n-k)}(w_\lambda)(1-|w_\lambda|^2)^{n}\right|^p\,.$$
Finally, taking into account that  $\{z_\lambda\}_\lambda$ is an $r$-lattice and incorporating a standard hyperbolic estimate \eqref{geometric hyperbolic estimates}, we find a constant $C$, depending on all the parameters, $n,r,b,\alpha$ and not on $g\in \Hol(\m D)$, such that
\begin{align*}
    \int_{\m{D}}&|G(z,z)|^p(1-|z|^2)^{np}\frac{dA(z)}{(1-|z|^2)^2}\leq \sum_{\lambda=1}^{\infty}\int_{D(z_\lambda,r)}\left|G(z,z)(1-|z|^2)^n\right|^{p}\frac{dA(z)}{(1-|z|^2)^2}\\
    &\leq N \sum_{\lambda=1}^{\infty}\left|\sum_{k=0}^{n-1}a_k\frac{(b)_k|w_\lambda|^{jk}\overline{w_\lambda}^kg^{(n-k)}(w_\lambda)(1-|w_\lambda|^2)^n}{(1-|w_\lambda|^{j+2})^{k}}\right|^p\int_{D(z_\lambda,r)}\frac{dA(z)}{(1-|z|^2)^2}\\
    &\leq NC \sum_{\lambda=1}^{\infty}\left|\sum_{k=0}^{n-1}a_k\frac{(b)_k|w_\lambda|^{jk}\overline{w_\lambda}^k}{(1-|w_\lambda|^{j+2})^{k}}g^{(n-k)}(w_\lambda)(1-|w_\lambda|^2)^n\right|^p\\
    &\leq NC\|T_{g,a}\|_p^p<\infty\,.
\end{align*}
   \par The next step is to use Lemma \ref{algebraic lhmma} for $f_k(z)=a_k(b)_k\overline{z}^kg^{(n-k)}(z)$. In the notation of Lemma \ref{algebraic lhmma}, we have proved that 
$$D_j\in L^p\left(\m{D},\frac{dA(z)}{(1-|z|^2)^2}\right)\qquad j=0,\dots,\lambda.$$

 Therefore, for as $a_\lambda\neq 0$, the function $a_\lambda g^{(n-\lambda)}(z)(1-|z|)^{n-\lambda}$ in $ \frac{1}{2}<|z|<1$ can be written as a linear combination of products of bounded functions and the functions $D_j$. Utilizing also that $g\in \Hol(\m D)$, we can show that,
$$\int_{\m{D}}|g^{(n-\lambda)}(z)|^p(1-|z|^2)^{(n-\lambda)p}\frac{dA(z)}{(1-|z|^2)^2}<\infty$$
proving that $g\in \mm{B}_p$.\\
 \end{proof}

\section{Proof of Theorem \ref{Theorem of paper Schatten class membership} for \texorpdfstring{$0<p<2$}{}}
\begin{proof}[Proof of Theorem \ref{Theorem of paper Schatten class membership} for \texorpdfstring{$0<p<2$}{}]
 Let $\alpha\geq -1$ and assume that $T_{g,a}\in S^p(A^2_\alpha)$. Equivalently, $ (T_{g,a})^*T_{g,a}\in S^{\frac{p}{2}}(A^2_\alpha)$. We fix an $r>0$ and let $\{z_\lambda\}_\lambda$ be an $r$-lattice in $\m{D}$. For a given $R>2r$ large enough, we can apply Lemma \ref{finite covering Lemma}, to partition $\{z_\lambda\}_\lambda$ into $N$ subsequences such that the hyperbolic distance between any two points in each subsequence is at least $R$, and thus allowing the hyperbolic discs in each subsequence being mutually disjoint. Let $\{\zeta_n\}_n$ be such a subsequence and consider the following sesquilinear form
 
        $$ S_g(f,h)=\sum_{m=1}^{\infty}\int_{D(\zeta_m,r)}T_{g,a}^{(n)}(f)(z)\overline{T_{g,a}^{(n)}(h)(z)}dA_{2n+\alpha}(z)\qquad f,h\in A^2_\alpha\,.$$
       We verify that $S_g$ is bounded sesquilinear form. To see this, we apply the Cauchy-Schwarz inequality, followed by the Littlewood Paley formula \eqref{Littlewood Paley Bergman} and finally the boundedness of $T_{g,a}$. Thus, we can find a constant $C$, depending only on $n$ and $\|T_{g,a}\|$, such that
       \begin{align*}
        |S_g(f,h)|^2
        &\leq \int_{\m{D}}|T_{g,a}^{(n)}(f)(z)|^2\sum_{m=1}^{\infty}\chi_{D(\zeta_m,r)}dA_{2n+\alpha}(z)\int_{\m{D}}|T_{g,a}^{(n)}(h)(z)|^2\sum_{m=1}^{\infty}\chi_{D(
        \zeta_m,r)}dA_{2n+\alpha}(z)\\
        &\leq \int_{\m{D}}|T_{g,a}^{(n)}(f)(z)|^2dA_{2n+\alpha}(z)\int_{\m{D}}|T_{g,a}^{(n)}(h)(z)|^2dA_{2n+\alpha}(z)\\
        &\leq C\|f\|_{A^2_\alpha}^2\|h\|_{A^2_\alpha}^2
        \end{align*}
        where $\chi$ denotes the characteristic function of a measurable set.

\par A standard application of Riesz Representation Theorem implies that $S_g$ induces a bounded operator, acting on $A^2_\alpha$, which we denote again by $S_g$ and its action on the inner product is given by
        $$\langle S_g(f),h\rangle_{A^2_\alpha}=S_g(f,h)\qquad f,h\in A^2_\alpha\,.$$
        Furthermore, the above estimates show that
        $$0\leq \langle S_g(f),f\rangle_{A^2_\alpha}\leq C\langle ((T_{g,a})^*T_{g,a}) f,f\rangle\qquad f\in A^2_\alpha.$$
        Hence, the characterization of the singular values through Rayleigh's equation \eqref{Min-Max Theorem} implies that $S_g\in S^{\frac{p}{2}}(A^2_\alpha)$ and
        $$\|S_g\|_{p/2}^{p/2}\leq C^{p/2}\|T_{g,a}\|_p^p\,.$$
        Now, we fix $\gamma>1+\frac{\alpha+2}{2}$ sufficiently large, $0\leq j\leq n-1$ and an orthonormal base $\{e_n\}_n\subset A^2_\alpha$. We construct the operator $B_j$ acting on $A^2_\alpha$, given by 
        
        $$B_j(f)(z)=\sum_{n=1}^{\infty}\langle f,e_n\rangle_{A^2_\alpha}(1-|z_n|^2)^{{\gamma+j}-\frac{\alpha+2}{2}}K_{j,z_n}(z)\qquad f\in A^2_\alpha\,. $$
        Once more, Lemma \ref{Operator Lhmma pgeq 2} and Parseval's identity imply that $B_j$ is bounded on $A^2_\alpha$. Consequently, by the ideal property of Schatten-von Neumann ideals \eqref{norm inequality when composed in Schatten},
        we have that the positive operator $F_j=B_j^*S_gB_j\in S^{\frac{p}{2}}(A^2_\alpha)$ and we can find a constant $C$ depending on $\gamma,n,p,r$ such that
        $$\|F_j\|_{p/2}^{p/2}\leq C\|T_{g,a}\|_p^p\,.$$ 
        The next step is to consider the following two compact operators
        \begin{align*}
            D_j(f)(z)&=\sum_{n=1}^{\infty}\langle F_j(e_n),e_n\rangle_{A^2_\alpha} \langle f,e_n\rangle_{A^2_\alpha} e_n\qquad f\in A^2_\alpha\\
            E_j&=F_j-D_j\,.
        \end{align*}
        We are going to estimate from below the Schatten norm of $D_j$. We have that
        \begin{align*}
            \|D_j\|_{p/2}^{p/2}&=\sum_{l=1}^{\infty}\langle F_j(e_l),e_l\rangle^{p/2}=\sum_{l=1}^{\infty} S_g(B_j(e_l),B_j(e_l))^{p/2}\\
            &=\sum_{l=1}^{\infty}\left(\sum_{m=1}^{\infty}\int_{D(\zeta_m,r)}\left|\sum_{k=0}^{\lambda}a_kg^{(n-k)}(z)(1-|\zeta_l|^2)^{\gamma+j-\frac{\alpha+2}{2}}K_{\zeta_l,j}^{(k)}(z)\right|^2dA_{2n+\alpha}(z)\right)^{\frac{p}{2}}\\
            &\geq \sum_{l=1}^{\infty}\left(\int_{D(\zeta_l,r)}\left|\sum_{k=0}^{\lambda}a_kg^{(n-k)}(z)(1-|\zeta_l|^2)^{\gamma+j-\frac{\alpha+2}{2}}K_{\zeta_l,j}^{(k)}(z)\right|^2dA_{2n+\alpha}(z)\right)^{\frac{p}{2}}\\
            &\geq C\sum_{l=1}^{\infty}\left(\int_{D(\zeta_l,r)}\left|\sum_{k=0}^{\lambda}a_kg^{(n-k)}(z)(|1-\overline{\zeta_l} z|)^{j}K_{\zeta_l,j}^{(k)}(z)\right|^2(1-|\zeta_l|^2)^{2\gamma-2+2n}dA(z)\right)^{\frac{p}{2}}
        \end{align*}
        where $C$ is a positive constant depending on $r$, obtained by hyperbolic estimates \eqref{geometric hyperbolic estimates}. For $0\leq j\leq \lambda$, set $f_j(z)=a_{j}g^{(n-j)}(z).$ Then, by the strength of Lemma \ref{Lemma of Arrousi on linear algebra} and multiple use of triangle inequalities, we can find a constant $C$, now depending on $n,p,r$, such that
        \begin{equation}\label{equation1}
            \sum_{k=0}^{\lambda}\|D_k\|_{p/2}^{p/2}\geq C\sum_{l=1}^{\infty}\left(\int_{D(\zeta_l,r)}\left|\frac{a_\lambda g^{(n-\lambda)}(z)}{(1-\overline{\zeta_l}z)^{\gamma+\lambda}}\right|^2(1-|\zeta_l|^2)^{2\gamma-2+2n}dA(z)\right)^{\frac{p}{2}}\,.
        \end{equation}
        The next step is to progress by induction and acquire the former equation for every $0\leq j\leq \lambda$, that is
        \begin{equation}\label{equation2}
            \sum_{k=0}^{\lambda}\|D_k\|_{p/2}^{p/2}\geq C\sum_{l=1}^{\infty}\left(\int_{D(\zeta_l,r)}\left|\frac{a_\lambda g^{(n-j)}(z)}{(1-\overline{\zeta_l}z)^{\gamma+j}}\right|^2(1-|\zeta_l|^2)^{2\gamma-2+2n}dA(z)\right)^{\frac{p}{2}}
        \end{equation}
        To do this, we observe that
        $$\sum_{k=0}^{\lambda-1}a_kg^{(n-k)}K_{\zeta_l,0}^{(k)}(z)=\sum_{k=0}^{\lambda}a_kg^{(n-k)}K_{\zeta_l,0}^{(k)}(z)-\frac{a_\lambda g^{(n-\lambda)}(z)}{(1-\overline{\zeta_l}z)^{\gamma+\lambda}}\overline{\zeta}_l^\lambda$$
        and therefore, the strength of \eqref{equation1} and the use of triangle inequalities, allow us to find a constant $C$, depending on $n,r,p,\gamma$ such that
        \begin{align*}
           \sum_{l=1}^{\infty}\left(\int_{D(\zeta_l,r)}\left|\sum_{k=0}^{\lambda-1}a_kg^{(n-k)}(1-|\zeta_l|^2)^{\gamma-\frac{\alpha+2}{2}}K_{\zeta_l,0}^{(k)}(z)\right|^2(1-|\zeta_l|^2)^{2n}dA_\alpha(z) \right)^{\frac{p}{2}}&\leq C\left(\sum_{k=0}^{\lambda}\|D_k\|_{p/2}^{p/2}\right)\,.
        \end{align*}
         Thus, we reduce the terms of the sum in each occurrence, finally reaching up to \eqref{equation2}. 
        \par Now, let $J$ be the index for which $\|D_j\|_{p/2}^{p/2}$ is maximized. Then, from hyperbolic estimates and the strength of \eqref{equation2}, we conclude that there exists a constant $C_1$, depending on the parameters $n,r,p,\gamma$, such that
        
        \begin{equation}\label{equation to maximimze}
        \|D_J\|_{p/2}^{p/2}\geq C_1\sum_{k=0}^{\lambda}\sum_{l=1}^{\infty}\left(\int_{D(\zeta_l,r)}\left|a_k g^{(n-k)}(z)\right|^2dA_{2(n-k)-2}(z)\right)^{\frac{p}{2}}\,.
        \end{equation}
        The next step is technical and involves the estimation from above of $\|E_J\|_{p/2}^{p/2}$. To do that, we employ the strength of Lemma \ref{0<p<2 Schatten norm estimation}. Subsequently, by the means of H\"older's inequality for series \cite[Theorem 3.3]{zhu2007operator}, and the estimate of \eqref{equation to use for Kiw} and triangle inequalities, we find a constant $C>0$ depending on $n,p$, such that
        \begin{align}
            \|E_J\|_{p/2}^{p/2}&\leq \sum_{\kappa =1}^{\infty}\sum_{\nu=1}^{\infty}|\langle E_J(e_\kappa),e_\nu\rangle|^{p/2}=\sum_{\kappa\neq \nu}|\langle T_J(e_\kappa),e_\nu\rangle|^{p/2}\nonumber\\
            &=\sum_{\kappa\neq \nu}| S_g(B_J(e_\kappa),B_J(e_\nu))|^{p/2}\nonumber\\
            &\leq C\sum_{k=0}^{\lambda}\sum_{\kappa\neq \nu}\sum_{m=1}^{\infty}\left(\int_{D(\zeta_m,r)}|G_{1,k}(z)|dA_{2n+\alpha}(z)\right)^{p/2}+\nonumber\\
            &\qquad\qquad+ C\sum_{\substack{k,s=0\\ k\neq s}}^{\lambda^2-\lambda}\sum_{\kappa\neq \nu}\sum_{m=1}^{\infty}\left(\int_{D(\zeta_m,r)}|G_{2,k,s}(z)|dA_{2n+\alpha}(z)\right)^{p/2}\nonumber\\
            &\doteqdot I_1+I_2\,.
        \end{align}
        where,
        \begin{align*}
           G_{1,k}&\doteqdot \frac{[(1-|\zeta_\kappa|^2)(1-|\zeta_\nu|^2)]^{\gamma+J-\frac{\alpha+2}{2}}}{[|1-\overline{\zeta_k}z||1-\overline{\zeta_\nu}z|]^{\gamma+J+k}}|a_k g^{(n-k)}(z)|^2\\
           &\\
            G_{2,k,s}&\doteqdot\displaystyle a_ka_sg^{(n-k)}(z)g^{(n-s)}(z)[(1-|\zeta_\kappa|^2)((1-|\zeta_\nu|^2)]^{\gamma+J-\frac{\alpha+2}{2}}K_{\zeta_\kappa,J}^{(k)}(z)K_{\zeta_\nu,J}^{(s)}(z)\,.
        \end{align*}

We focus our attention to estimate $I_1$. To start with, we observe that with the strength of hyperbolic estimates \eqref{geometric hyperbolic estimates}, we find a constant $C$ depending only on $r$, such that 
\begin{align}
    \sum_{\kappa\neq \nu}\sum_{m=1}^{\infty}&\left(\int_{D(\zeta_m,r)}\frac{[(1-|\zeta_\kappa|^2)((1-|\zeta_\nu|^2)]^{\gamma+J-\frac{\alpha+2}{2}}}{[|1-\overline{\zeta_\kappa}z||1-\overline{\zeta_\nu}z|]^{\gamma+J+k}}|a_kg^{(n-k)}(z)|^2dA_{2n+\alpha}(z)\right)^{\frac{p}{2}}\leq \nonumber\\
    &\leq C\sum_{\kappa\neq \nu}\sum_{m=1}^{\infty}\frac{[(1-|\zeta_\kappa|^2)((1-|\zeta_\nu|^2)]^{\frac{p(\gamma+J)}{2}-\frac{p(\alpha+2)}{4}}}{[|1-\overline{\zeta_\kappa}\zeta_m||1-\overline{\zeta_\nu}\zeta_m|]^{\frac{p(\gamma+J)}{2}+\frac{pk}{2}}}\left(\int_{D(\zeta_m,r)}|a_kg^{(n-k)}(z)|^2dA_{2n+\alpha}(z)\right)^{\frac{p}{2}}\nonumber\\
    &=C\sum_{m=1}^{\infty}\left(\int_{D(\zeta_m,r)}|a_kg^{(n-k)}(z)|^2dA_{2n+\alpha}(z)\right)^{\frac{p}{2}}\sum_{\kappa\neq \nu}\frac{[(1-|\zeta_\kappa|^2)((1-|\zeta_\nu|^2)]^{\frac{p(\gamma+J)}{2}-\frac{p(\alpha+2)}{4}}}{[|1-\overline{\zeta_\kappa}\zeta_m||1-\overline{\zeta_\nu}
    \zeta_m|]^{\frac{p(\gamma+J)}{2}+\frac{pk}{2}}}\label{equation to return back after kernel estimation}
\end{align}
where the final equality is due to Fubini's Theorem. The next step involves the estimation of the double sum ${\kappa\neq \nu}$ appearing in \eqref{equation to return back after kernel estimation}. Since, this is  a well-established result, we only sketch the argument here, highlighting the key steps. For complete details, we refer the reader to the proof of \cite[Theorem 7.16]{zhu2007operator}. Since the function $h_{w}(z)= \frac{1}{1-\overline{w}z}$ is analytic, an application of estimates \eqref{geometric hyperbolic estimates} and \eqref{subharmonicity hyperbolic estimate} implies that the terms in the sum are bounded above by an integral over the set $D(\zeta_\nu,r)\times D(\zeta_\kappa,r)$. Since the discs $D(\zeta_\nu,r)$ and $D(\zeta_\kappa,r)$ are mutually disjoint, it follows that
$$\bigcup_{\kappa\neq \nu} D(\zeta_\nu,r)\times D(\zeta_\kappa,r)\subset A_R=\{(z,w)\in\m{D}\times \m{D}\colon \beta(z,w)\geq R-2r\}.$$
Thus, we estimate further the sum by enlarging the region of integration to $A_R$. Employing the fact that $A_R$ is M\"obius invariant, we use the change of variables $(z,w)=(\phi_{\zeta_m}(u),\phi_{\zeta_m}(w))$, to find a positive constant $C$ depending on $p,b,\alpha$ such that
\begin{equation}\label{equation to put back}
    \sum_{\kappa\neq \nu}\frac{[(1-|\zeta_\kappa|^2)((1-|\zeta_\nu|^2)]^{\frac{p(\gamma+J)}{2}-\frac{p(\alpha+2)}{4}}}{[|1-\overline{\zeta_\kappa}\zeta_m||1-\overline{\zeta_\nu}
    \zeta_m|]^{\frac{p(\gamma+J)}{2}+\frac{pk}{2}}}\leq C|A_R|(1-|\zeta_m|^2)^{-pk-\frac{p(\alpha+2)}{2}}
\end{equation}
    where $|A_R|$ is the Lebesgue volume of the set $A_R$.
    \par We combine the estimates of \eqref{equation to put back} and \eqref{equation to return back after kernel estimation}, take into account hyperbolic estimates \eqref{geometric hyperbolic estimates} and finally sum over all $0\leq k\leq \lambda$. Hence, we are able to find a constant $C_1>0$, depending on the parameters $n,r,p,\gamma,\alpha$ such that,
    \begin{align}
        I_{1}\leq C_{1}|A_R|\sum_{k=0}^{\lambda}\sum_{m=1}^{\infty}\left(\int_{D(\zeta_m,r)}|a_kg^{(n-k)}(z)|^2dA_{2(n-k)-2}(z)\right)^{\frac{p}{2}}\,.\label{equation of I1k}
    \end{align}
   
    Our next objective is to estimate $I_2$. The key estimate here is to apply the special case of arithmetic mean-geometric mean (AM-GM) inequality, combined with \eqref{equation to use for Kiw}. In particular, for each $0\leq k,s\leq \lambda$, such that $k\neq s$, we have that
    \begin{align}
        |a_{k}g^{(n-k)}(z)K_{\zeta_\kappa,J}^{(k)}(z) a_{s}g^{(n-s)}(z)K_{\zeta_\nu,J}^{(s)}(z)|&\leq \frac{1}{2}(|a_kg^{(n-k)}(z)K_{\zeta_\kappa,J}^{(k)}(z)|^2+|a_sg^{(n-s)}(z)K_{\zeta_\nu,J}^{(s)}(z)|^2)\nonumber\\
        &\leq C\left(\frac{|a_kg^{(n-k)}(z)|^2}{|1-\overline{\zeta_\kappa}z|^{2\gamma+2J+2K}}+\frac{|a_sg^{(n-s)}(z)|^2}{|1-\overline{\zeta_\nu}z|^{2\gamma+2J+2s}}\right)\,. \label{equation G2kh}
    \end{align}
    Thus, we can split the summation over different values of $k$ and $s$. As a consequence of \eqref{equation G2kh}, to estimate $I_2$, it suffices to estimate the following quantities 
    \begin{align}
        \sum_{\kappa\neq \nu}\sum_{m=1}^{\infty}\left(\int_{D(\zeta_m,r)}\frac{[(1-|\zeta_\kappa|^2)(1-|\zeta_\nu|^2)]^{\gamma+J-\frac{\alpha+2}{2}}}{|1-\overline{\zeta_\kappa}z|^{2\gamma+2J+2k}}|a_kg^{(n-k)}(z)|^2dA_{2n+\alpha}(z)\right)^{p/2}&\label{equation1withk}\\
        \sum_{\kappa\neq \nu}\sum_{m=1}^{\infty}\left(\int_{D(\zeta_m,r)}\frac{[(1-|\zeta_\kappa|^2)(1-|\zeta_\nu|^2)]^{\gamma+J-\frac{\alpha+2}{2}}}{|1-\overline{\zeta_\nu}z|^{2\gamma+2J+2s}}|a_hg^{(n-s)}(z)|^2dA_{2n+\alpha}(z)\right)^{p/2}\label{equation2withh}&
    \end{align}
    The estimate of both \eqref{equation1withk} and \eqref{equation2withh} follows a similar approach as the one in \eqref{equation to return back after kernel estimation}. Due to symmetry, we provide the details only for \eqref{equation1withk}. In fact, using hyperbolic estimates \eqref{geometric hyperbolic estimates}, we can find a constant $C$ depending on $r$, such that 
    \begin{align*}
        &\sum_{\kappa\neq \nu}\sum_{m=1}^{\infty}\left(\int_{D(\zeta_m,r)}\frac{[(1-|\zeta_\kappa|^2)(1-|\zeta_\nu|^2)]^{\gamma+J-\frac{\alpha+2}{2}}}{|1-\overline{\zeta_\kappa}z|^{2\gamma+2J+2k}}|a_kg^{(n-k)}(z)|^2dA_{2n+\alpha}(z)\right)^{p/2}\leq \\
        &\leq C\sum_{k\neq\nu}\frac{[(1-|\zeta_\kappa|^2)(1-|\zeta_\nu|^2)]^{p(\gamma+J)/2-\frac{p(\alpha+2)}{4}}}{|1-\overline{\zeta_\kappa}\zeta_m|^{p\gamma+pJ+pk}}\sum_{m=1}^{\infty}\left(\int_{D(\zeta_m,r)}|a_kg^{(n-k)}(z)|^2dA_{2n+\alpha}(z)\right)^{p/2}
    \end{align*}
    Then arguing as in the part of the estimations of $I_{1}$, we can find a constant $C$, depending on $r,\gamma,\alpha$, such that
    \begin{align*}
    \sum_{k\neq\nu}&\frac{[(1-|\zeta_\kappa|^2)(1-|\zeta_\nu|^2)]^{p(\gamma+J)/2-\frac{p(\alpha+2)}{4}}}{|1-\overline{\zeta_\kappa}\zeta_m|^{p\gamma+pJ+pk}}\leq C|A_R|(1-|\zeta_m|^2)^{-pk-p\frac{\alpha+2}{2}}
    \end{align*}
    Finally, we conclude that
    \begin{align*}
        &\sum_{\kappa\neq \nu}\sum_{m=1}^{\infty}\left(\int_{D(\zeta_m,r)}\frac{[(1-|\zeta_\kappa|^2)(1-|\zeta_\nu|^2)]^{\gamma+J-\frac{\alpha+2}{2}}}{|1-\overline{\zeta_\kappa}z|^{2\gamma+2J+2k}}|a_kg^{(n-k)}(z)|^2dA_{2n+\alpha}(z)\right)^{p/2}\leq \\
        &\leq C|A_R|\sum_{m=1}^{\infty}\left(\int_{D(\zeta_m,r)}|a_kg^{(n-k)}(z)|^2dA_{2(n-k)-2}(z)\right)^{p/2}
    \end{align*}
    Consequently, we can find a constant $C_2$, depending on $n,r,\gamma,\alpha,p$, such that
    \begin{equation}\label{I2ks estimate}
    I_2 \leq C_2|A_R|\sum_{k=0}^{\lambda}\sum_{m=1}^{\infty}\left(\int_{D(\zeta_m,r)}|a_kg^{(n-k)}(z)|^2dA_{2(n-k)-2}(z)\right)^{\frac{p}{2}}\,.
    \end{equation}
    So, we combine \eqref{I2ks estimate} with \eqref{equation of I1k} to conclude that there is a constant $C_2$, depending on $n,r,\gamma,\alpha,p$, such that 
    \begin{equation}\label{finalequationofEJ}
        \|E_J\|_{p/2}^{p/2}\leq C_2|A_R|\sum_{k=0}^{\lambda}\sum_{m=1}^{\infty}\left(\int_{D(\zeta_m,r)}|a_kg^{(n-k)}(z)|^2dA_{2(n-k)-2}(z)\right)^{\frac{p}{2}}\,.
    \end{equation}

    \par Now, as $C_1,C_2$ depend only on the fixed parameters $n,r,p,\gamma,\alpha$, we can choose $R>0$, large enough so that $|A_R|<\frac
    {C_1}{C_2}$, in other words, $C_1-|A_R|C_2>0$. Then the strength of \eqref{equation to maximimze} and 
    \eqref{finalequationofEJ} imply that
    \begin{align*}
        \|F_J\|_{p/2}^{p/2}&\geq |\|D_J\|_{p/2}^{p/2}-\|E_J\|_{p/2}^{p/2}|\\
        &\geq (C_1-|A_R|C_2)\sum_{k=0}^{\lambda}\sum_{m=1}^{\infty}\left(\int_{D(\zeta_m,r)}|a_kg^{(n-k)}(z)|^2dA_{2(n-k)-2}(z)\right)^{\frac{p}{2}}\\
        &\geq (C_1-|A_R|C_2)\sum_{m=1}^{\infty}\left(\int_{D(\zeta_m,r)}|a_{\lambda}g^{(n-\lambda)}(z)|^2dA_{2(n-\lambda)-2}(z)\right)^{\frac{p}{2}}\,.
    \end{align*}
    Taking into account all aforementioned estimates, we finally find a constant $C$, depending on the fixed parameters while being independent of the given function $g$, such that
    $$\sum_{m=1}^{\infty}\left(\int_{D(\zeta_m,r)}|a_{\lambda}g^{(n-\lambda)}(z)|^2dA_{2(n-\lambda)-2}(z)\right)^{\frac{p}{2}}\leq C\|T_{g,a}\|_{p/2}^{p/2}\,.$$
    Since the above inequality holds for each one of the $N$ subsequences of $\{\zeta_m\}_m$, we obtain that
    $$\sum_{m=1}^{\infty}\left(\int_{D(z_m,r)}|a_{\lambda}g^{(n-\lambda)}(z)|^2dA_{2(n-\lambda)-2}(z)\right)^{\frac{p}{2}}\leq CN\|T_{g,a}\|_{p/2}^{p/2}\,.$$
    Now, the reader can follow the lines of the proof of \cite[Corollary 1]{DuLiQu2022generalized} to show that the above summability condition implies that
    $$\int_{\m{D}}|g^{(n-\lambda)}(z)|^p(1-|z|^2)^{p(n-\lambda)}\frac{dA(z)}{(1-|z|^2)^2}<\infty\,.$$
    So, if $p>\frac{1}{n-\lambda}$, then we equivalently have that $g\in \mm{B}_p$ and if $p\leq \frac{1}{n-\lambda}$, standard estimates show that $g^{(n-\lambda)}\equiv 0$, implying that $T_{g,a}\equiv 0$.\\
\end{proof}
 
\bibliography{bibliography}
\bibliographystyle{plain}
\end{document}